\documentclass[amssymb,10pt]{article}
\usepackage[pagewise]{lineno}
\usepackage{graphicx}
\usepackage{indentfirst}
\usepackage{amsmath}
\usepackage{color}

\usepackage{amssymb}
\usepackage{amsthm}
\usepackage{titletoc}
\usepackage{amssymb}
\usepackage{amsxtra}
\usepackage{mathrsfs}
\usepackage{amsmath,amstext,amsfonts,verbatim, url}
\usepackage{epstopdf}
\usepackage{subfig}
\usepackage{psfrag}

\usepackage{tikz}

\usepackage{xy}

\usepackage[colorlinks=true]{hyperref}
\usepackage{amsfonts}
\usepackage{mathrsfs}
\usepackage{amsthm,amsmath,amssymb,mathrsfs,amsfonts,graphicx,graphics,latexsym,exscale,cmmib57,dsfont,bbm,amscd,ulem}
\usepackage{color,soul}
\newtheorem{thm}{Theorem}[section]
\newtheorem{lem}[thm]{Lemma}

\newtheorem{cor}[thm]{Corollary}

\def\R{{\mathbb R}}
\def\Z{{\mathbb Z}}

\def\T{{\mathbb T}}

\def\bb{\begin}
\def\be{\begin{equation}}
	\def\ee{\end{equation}}
\def\bea{\begin{eqnarray}}
	\def\eea{\end{eqnarray}}
\def\beaa{\begin{eqnarray*}}
	\def\eeaa{\end{eqnarray*}}

\def\ifl{\iffalse}

\def\bb{\begin}
           \def\ea{\end{array}}
          \def\ec{\end{center}}
     \def\ed{\end{description}}
\def\be{\bb{equation}}        \def\ee{\end{equation}}
\def\bea{\bb{eqnarray}}       \def\eea{\end{eqnarray}}
\def\beaa{\bb{eqnarray*}}     \def\eeaa{\end{eqnarray*}}
 \def\et{\end{thebibliography}}

       \def\nt{\noindent}

\begin{document}

\title{Notes on any given number of non-hyperbolic physical measures of some partially hyperbolic diffeomorphism.}

\date{}
\maketitle

{\center
Hangyue Zhang

\smallskip

Department of Mathematics

Nanjing University

Nanjing 210093, China

\smallskip

\smallskip


\smallskip

}

\begin{abstract}
In this paper, we provide an example of a partially hyperbolic diffeomorphism with any finite number of physical measures when some Lyapunov exponent is 0 on the center.
\end{abstract}

\smallskip
\smallskip
\smallskip
\smallskip
\smallskip
\smallskip

Let $M$ be a smooth connect compact Riemannian manifold and $f$ be a diffeomorphism on $M$. An invariant measure $\mu$ of $f$ is called {\it physical measure} if the {\it basin} of $\mu$ denoted by 
\begin{center}
	$B(\mu)=\{x\in M: \mu=\lim_{n\to+\infty} \frac{\sum_{0\le j\le n-1}\delta_{f^j(x)}}{n}$ in $weak^*$ topology$\}$
\end{center}
has positive Lebesgue measure. The physical measure was first discovered by Sinai, Ruelle and Bowen when they studied Anosov systems and Axiom $A$ attractors in 1970's\cite{Bow,th,am,gi} and can be used to observe the asymptotic behavior of orbits in the sense of Lebesgue measure.  With the development of the discipline of dynamical systems, there have been many achievements in the study of physical measures with non-uniform hyperbolicity such as \cite{BV,DVY,ja1,ja3,ja2,cv,and,CM}. However, these achievements are all research on physical measures in the sense of {\it hyperbolic measure} denoted by an invariant measure whose Lyapunov exponents are all non-zero.  Then, naturally, some people will ask, can non-hyperbolic measures of some partially hyperbolic diffeomorphism be physical measures? If so, how much can there be? Do they possess basin covering property? Then, this article provides a positive answer to those questions and provide any finite number of physical measures on some partially hyperbolic diffeomorphism with basin covering property.

A diffeomorphsim $f$ is {\it partially hyperbolic}  if tangent bundle $TM=E^u\oplus F$ satisfying that there exists $C>0,\lambda\in(0,1)$ and $Df(E^u)=E^u, Df(F)=F$ such that for any unit vectors $v^u\in E^u, v\in F$,
\[\frac{\|Df^n(v)\|}{\|Df^n(v^u)\|}\le C\lambda^n, \|Df^{-n}(v^u)\|\le  C\lambda^n.\]

Partially hyperbolic diffeomorphisms play an important role in non-uniform hyperbolicity. It extends the definition of uniform hyperbolicity and is suitable for most diffeomorphsim.

Let $A:\T^2\to\T^2$ be a hyperbolic linear automorphism 
\( \begin{pmatrix}
	2  &  1\\
	
	1 &  1
\end{pmatrix} \), $\alpha$ be a irrational number and $\tau_\alpha:S^1\to S^1$ be an  irrational rotation on circle $S^1$ whose evolutionary form is \( \tau_\alpha(x)=x+\alpha \) mod 1.

\begin{lem}\label{zhuyao}
	Denote $f=A\times\tau_\alpha:\T^2\times S^1\to\T^2\times S^1$. Then $f$ admit the unique physical measure which consists with the Lebesgue measure  of $\T^2\times S^1$ denoted by $m_{\T^2\times S^1}$.
\end{lem}
\begin{proof}
	Notice that $f$ is induced the map $\tilde{f}(x,y,z)=(2x+y,x+y,z+\alpha)$ on $\R^3$. Consider the measurable function $\phi\in L^2(m_{\T^2\times S^1})$ and $\phi\circ f=\phi$ a.e-$m_{\T^2\times S^1}$. Then we can assume that $\phi$ is induced by $\tilde{\phi}$ which also can be called the lift of $\phi$. Thus, $\tilde{\phi}\circ\tilde{f}=\tilde{\phi}$. By Fourier series of $\tilde{\phi}$,
	\[\tilde{\phi}(x,y,z)=\sum_{m,n,k}a_{m,n,k}\kappa_{m,n,k}(x,y,z)\]
	where $\kappa_{m,n,k}(x,y,z)=e^{2\pi i(mx+ny+kz)}$ and $a_{m,n,k}$ is constant. It follows that
	\begin{align*}
		\begin{split}
			\kappa_{m,n,k}\circ\tilde{f}(x,y,z)&=e^{2\pi i((2x+y)m+(x+y)n+(z+\alpha)k)} \\
			& =e^{2\pi i\alpha k}\cdot e^{2\pi i((2m+n)x+(m+n)y+kz)} \\
			& =e^{2\pi i\alpha k}\kappa_{2m+n,m+n,k}.
		\end{split}
	\end{align*}
	By the uniqueness of Fourier expansion coefficients and $\tilde{\phi}\circ\tilde{f}=\tilde{\phi}$, 
	\[e^{2\pi i\alpha k}a_{m,n,k}=a_{2m+n,m+n,k}.\]
	Claim that $a_{m,n,k}=0$ for any $(m,n,k)\neq(0,0,0)$. Then $\tilde{\phi}=a_{0,0,0}$ a.e-Lebesgue. It follows that $m_{\T^2\times S^1}$ is an ergodic measure and  there is a $m_{\T^2\times S^1}$-full measure set contained in  basin $B(m_{\T^2\times S^1})$.  Next, We only need to prove the claim.
	
	Notice that $(m,n,k)\neq(0,0,0)$ can be divided into two situations:
	\begin{itemize}
		\item $m$ and $n$ are both zero, $k$ is not zero;
		\item $m$ and $n$ are not all zero.
	\end{itemize}
	For the first item, \( e^{2\pi i\alpha k}a_{0,0,k}=a_{0,0,k} \). Since $\alpha$ is a irrational number, \( a_{0,0,k}=0 \) for all $k\neq 0$.  Arguing by contradiction for the second item, assume that there exists some $a_{m,n,k}\neq0$ for some $m$ or $n$ not zero. This implies that there exist infinite coefficients whose modulus coincide with $|a_{m,n,k}|$. This would contradicts the fact that $|a_{m,n,k}|\to 0$ when $|m|+|n|+|k|\to\infty$. Then we complete the proof of the claim.
\end{proof}

Recall that irrational numbers $\alpha_1,\alpha_2,...,\alpha_\ell$ are {\it rationally independent} if any inters $C_1,C_2,...,C_\ell$ satisfying that
$$
C_1\alpha_1+C_2\alpha_2+...+C_\ell\alpha_\ell=0 \ {\rm mod 1}
$$
can only all be $0$. Then the constructed $f$ can also be $A\times\tau_{\alpha_1}\times\tau_{\alpha_2}\times\cdots\times\tau_{\alpha_\ell}$ on $\T^2\times S^1_1\times\cdots\times S^1_\ell$ as long as $\alpha_1,\alpha_2,...,\alpha_\ell$ are  rationally independent where $S^1_i=S^1$. We can obtain the corresponding same result by the above similar reason.

Recall that a homeomorphism $H$ is called {\it topologically transitive} if for any open sets $U,V$ of $M$, there exists some $N\in\mathbb{N}^+$ such that $H^N(U)\cap V\neq\emptyset$.
Then, due to the ergodicity of the Lebesgue measure(Lemma~\ref{zhuyao}), we can immediately obtain the following corollary.
\begin{cor}
	f is topologically transitive.
\end{cor}

Next, we construct the partially hyperbolic diffeomorphsim $g$ on $\T^2\times S^1\times S^1$ and prove following theorem. Notice that there always exists diffeomorphisms such that it doesn't admit physical measures. We can check that by considering  direct product of any diffeomorphism with identity on  Riemannian manifold. So we do not have any theoretical results regarding the 0-Lyapunov exponent center.

\smallskip
\smallskip
\smallskip

\begin{thm}\label{gouzao}
	For any $\ell\in\Z^+$, there exists some partially hyperbolic diffeomorphism $g$ on $\T^2\times S^1\times S^1$ such that $g$ admits $\ell$ physical measures such that the union of basins of $\ell$-physical measures has full-Lebesgue measure and each physical measure is non-hyperbolic.
\end{thm}

\smallskip
\smallskip
\smallskip

This results also hold when partially hyperbolic diffeomorphism admits multiple centers as long as we adjust the construction of $g$.

\nt Construction of $g$:

For any $\ell$, let $h: S^1 \to S^1$ be a $C^1$ Morse-Smale diffeomorphism and $\{u_1,\dots,u_{2\ell}\}$ be the $2\ell$-ordered points on the circle satisfying:
\begin{itemize}
	\item the non-wandering set of $h$ is $\{u_1,\dots,u_{2\ell}\}$;
	\item denote $u_{2\ell+1}=u_1$,  $u_{2i}$ is a sink and $u_{2i+1}$ is a source  for each $i\in\{1,2,...,\ell\}$.
\end{itemize} 

Let $g=f\times h$. Then, we can prove our theorem.

\begin{proof}[Proof of Theorem~\ref{gouzao}]
By Fubini's theorem, it is easy to check that $B(m_{\T^2\times S^1})\times (S^1\backslash\{u_{2i+1}:i=1,2,...,\ell\})$ has Lebesgue-full measure. Next, we will prove that $B(m_{\T^2\times S^1\times\{u_{2i}\}})=B(m_{\T^2\times S^1})\times (u_{2i-1,2i+1})$. If this result is correct, then each $m_{\T^2\times S^1\times\{u_{2i}\}}$ is a physical measure and the union of  basins of all $m_{\T^2\times S^1\times\{u_{2i}\}}$ is full-Lebesgue measure. It is clear that $g$ is partially hyperblic and any invariant measure of $g$ is non-hyperbolic. Thus we just need to prove the result we just mentioned.

It is obvious that $B(m_{\T^2\times S^1\times\{u_{2i}\}})\subset B(m_{\T^2\times S^1})\times (u_{2i-1,2i+1})$ since each $u_{2i}$ is a sink of $h$.

 Because each $u_{2i}$ is a sink of $h$ again, $\lim_{n\to+\infty}d(g^n(x,y,z),g^n(x,y,u_{2i}))=0$ for any $(x,y,z)\in B(m_{\T^2\times S^1})\times (u_{2i-1,2i+1})$.  Similar to Lemma~\ref{zhuyao}, any point in $B(m_{\T^2\times S^1\times\{u_{2i}\}})\cap \T^2\times S^1\times\{u_{2i}\}$ can be written as  $(x,y,u_{2i})$ in $B(m_{\T^2\times S^1})\times \{u_{2i}\}$.  For any continuous function $\psi$ on $\T^2\times S^1\times S^1$, due to the uniform continuity and boundedness of continuous functions in compact spaces, then for any $\varepsilon>0$, there exist $\delta(\varepsilon)$ and $P$ such that any $a,b$ satisfying that $d(a,b)<0$, $|\psi(a)-\psi(b)|<\varepsilon$ and $\sup|\psi|\le  P$ hold here.
  In particular, for such $\delta(\varepsilon)$, there exists $N_\delta>0$ such that for any $n\ge N_\delta$, $d(g^n(x,y,z),g^n(x,y,u_{2i}))<\delta(\varepsilon)$. This implies that for $n\ge N_\delta$, $|\psi(g^n(x,y,z))-\psi(g^n(x,y,u_{2i}))|<\varepsilon$. Then

\begin{align*}
	\begin{split}
		\frac{\sum_{0\le i\le n-1}\psi(g^n(x,y,z))}{n} & \le\frac{\sum_{0\le i\le n-1}\psi(g^n(x,y,u_{2i}))}{n}+\frac{N_\delta P+(n-N_\delta)\varepsilon}{n} \\
		& \le\frac{\sum_{0\le i\le n-1}\psi(g^n(x,y,u_{2i}))}{n}+\frac{N_\delta P}{n}+\varepsilon
	\end{split}
\end{align*}
Notice that $\frac{N_\delta P}{n}$ can be arbitrarily small when $n$ is sufficiently large. It follows that 
\[m_{\T^2\times S^1\times\{u_{2i}\}}=\lim_{n\rightarrow+\infty}\frac{\sum_{0\le i\le n-1}\delta_{g^i(x,y,u_{2i})}}{n}=\lim_{n\rightarrow+\infty}\frac{\sum_{0\le i\le n-1}\delta_{g^i(x,y,z)}}{n}.\]
In other words, $B(m_{\T^2\times S^1})\times (u_{2i-1,2i+1})\subset B(m_{\T^2\times S^1\times\{u_{2i}\}})$.
\end{proof}

\smallskip

E-mail address: zhanghangyue@nju.edu.cn

\end{document}